\documentclass[11pt,reqno]{amsart}
\usepackage{amsmath,amssymb,physics}
\usepackage[mathscr]{eucal}
\usepackage[colorlinks]{hyperref}
\usepackage[capitalize,nameinlink]{cleveref}
\hypersetup{urlcolor=blue, citecolor=red, linkcolor=blue}
\usepackage{nicematrix} %used in Introduction
\newcommand{\ran}{\operatorname{ran}}

\newcommand{\spncl}{\overline{\operatorname{span}}}

\newcommand{\C}{\mathbb{C}}
\newcommand{\Z}{\mathbb{Z}}
\newcommand{\R}{\mathbb{R}}

\newcommand{\D}{{\mathbb D}}

\newcommand{\Dc}{\overline{\mathbb D}}

\newcommand{\B}{\mathscr{B}}
\renewcommand{\S}{\mathscr S}
\renewcommand{\H}{\mathcal{H}}
\newcommand{\K}{\mathcal{K}}
\newcommand{\E}{\mathcal{E}} %for Hilbert space
\newcommand{\M}{\mathcal{M}}
\newcommand{\Q}{\mathcal Q}

\newcommand{\BH}{\mathscr{B}(\mathcal{H})}
\newcommand{\z}{\mathbf{z}}

\newcommand{\V}{\mathscr V}
\newcommand{\Tc}{\mathscr T}
\newcommand{\MzE}{M_{\z^\E}}
\newcommand{\CE}{\mathcal{C}_\E}
\newcommand{\HDE}{H^2(\D,\E)}
\newcommand{\Hinf}{H^\infty(\D,\B(\E))}

\newcommand{\X}{\mathfrak{X}}
\newcommand{\Wc}{\mathfrak W}
\newcommand{\Zc}{\mathfrak{Z}}

\newcommand{\la}{ \langle }
\newcommand{\ra}{\rangle}

\DeclareMathOperator{\slim}{s\!\cdot\! lim}

\newtheorem{thm}{Theorem}[section]

\newtheorem{corollary}[thm]{Corollary}
\newtheorem{lemma}[thm]{Lemma}

\newtheorem{definition}[thm]{Definition}
\newtheorem{remark}[thm]{Remark}

\newtheorem{example}[thm]{Example}

\numberwithin{equation}{section}

\def\textmatrix#1&#2\\#3&#4\\{\bigl({#1 \atop #3}\ {#2 \atop #4}\bigr)}
\def\dispmatrix#1&#2\\#3&#4\\{\left({#1 \atop #3}\ {#2 \atop #4}\right)}
\numberwithin{equation}{section}

\def\textmatrix#1&#2\\#3&#4\\{\bigl({#1 \atop #3}\ {#2 \atop #4}\bigr)}
\def\dispmatrix#1&#2\\#3&#4\\{\left({#1 \atop #3}\ {#2 \atop #4}\right)}
\usepackage{fancyhdr}
\begin{document}
	
	\title[On pure contractive semigroups]{On pure contractive semigroups}

	\author{Shubham Rastogi}%add later

    \author{Vijaya Kumar U}%add later
	
		\newcommand{\Addresses}{{% additional braces for segregating \footnotesize
			\bigskip
			\footnotesize
			
		S. ~Rastogi, \textsc{Department of Mathematics, Indian Institute of Technology Bombay, Mumbai 400076, India.}  \par\nopagebreak \textit{E-mail address}: \texttt{shubhamr@math.iitb.ac.in}
		
		\medskip
	
	 Vijaya Kumar U, \textsc{Department of Mathematics, Indian Institute of Science, Bangalore 560012.}\par\nopagebreak	\textit{E-mail address}: \texttt{vijayak@iisc.ac.in}.

	}}
	
	\maketitle
	
	\renewcommand{\thefootnote}{\fnsymbol{footnote}}
	
	
	\begin{abstract}
		We find the commutant of a pure contractive semigroup on a Hilbert space. We demonstrate that any tuple of doubly commuting pure contractive semigroups can be dilated to a tuple of doubly commuting pure isometric semigroups. En route, we obtain a complete model for the tuples of doubly commuting isometric semigroups. 		
	\end{abstract}
	
	\section{Introduction}\label{Sec-Intro}
	It is well known that any contraction $T$ on a Hilbert space has a unique minimal isometric dilation. If $T$ is a pure contraction, then its minimal isometric dilation is the multiplication operator $M_{\z^{\mathcal D_{T^*}}},$ by the coordinate function $\z^{\mathcal D_{T^*}}$ on the vector-valued Hardy space $H^2(\D,\mathcal D_{T^*}),$ where $\mathcal D_{T^*}$ is the  Hilbert space $\overline{\ran } (I-TT^*).$ This means there exists a ${(M_{\z^{\mathcal D_{T^*}}})}^*$-invariant subspace $\Q$ in $H^2(\D,\mathcal D_{T^*})$ such that $T$ is unitarily equivalent to the compression $P_\Q M_{\z^{\mathcal D_{T^*}}}|_\Q$, and $\spncl\{{(M_{\z^{\mathcal D_{T^*}}})}^n(\Q):n\in \Z_+\}=H^2(\D,\mathcal D_{T^*})$; see \cite{Nagy-Foias-paper, Nagy-Foias, Bhat-Bhattacharyya}. Analogously, any one-parameter pure contractive semigroup is unitarily equivalent to  a compression of the right shift semigroup $\S^\E$ on a ${(\S^\E)}^*$-invariant subspace, for some Hilbert space  $\E$ (see \cref{Sec-model-pure-cont-sgp} for the details). 
	 Let $\Hinf$ denote the algebra of all bounded analytic functions from the unit disc $\D$ to  $\B(\E)$ (the set of all bounded linear operators on $\E$). We simply write $H^\infty(\D)$ instead of $H^\infty(\D,\B(\C)).$ The result by Sarason in \cite{Sarason} states that any contraction which  commutes with   $\MzE$ (the multiplication operator by the coordinate function $\z^\E$)  on $\HDE$ is given by $M_\psi$, for some $\psi\in \Hinf$ with $\norm{\psi}_\infty \leq 1$. Cooper's result (\cite{Cooper}), shows that the role of $\MzE$ in the Wold decomposition of an isometry is played by the right shift semigroup for a semigroup of isometries. A result analogous to Sarason's commutant theorem for the right shift semigroup is established in \cite{Fact}. 
	
	%Let $T$ be a contraction on a Hilbert space $\H$, and let $(V,\K)$ be an isometric dilation of $(T,\H)$. If $R$ is any contraction on $\H$ that commutes with $T$, the commutant lifting theorem of Sz.Nagy and Foias asserts the existence of a contraction $S$ on $\K$ satisfying the conditions:\begin{equation}	VS=SV, S^*|_\H=R^* \text{ and }\norm{S}=\norm{R}. 	\end{equation}
The contractive commutant of a pure contraction can be described using the commutant lifting theorem (\cite{Nagy-Foias-CLT,Nagy-Foias}) and the Sarason's result (\cite{Sarason}) as follows: Let $\Q$ be an invariant subspace of ${(\MzE)}^*$ in $H^2(\D,\E).$ Then, for the pure contraction $P_\Q \MzE|_\Q$, its contractive commutant is given by:
\begin{equation}\label{pure-cont-comm}
(P_\Q \MzE|_\Q)'=\left\{ P_\Q M_{\psi}|_\Q:\psi\in \Hinf,\norm{\psi}_\infty\le 1\text{ and } M_\psi^*(\Q)\subseteq \Q\right\}.
\end{equation}
 In this note, we describe all the contractive semigroups which commute with a pure contractive semigroup. As a consequence, we get a model for the tuples of commuting contractive semigroups, where one of them is pure (see \cref{Sec-Comm-pure-cont-sgp}). 
 
 Ando's theorem (\cite{Ando}) states that any pair of commuting contractions possesses a dilation to a pair of commuting isometries. Such a dilation does not exist  for $n$-tuples of commuting contractions, in general for $n\ge 3$; see \cite{Var}. However, under certain additional assumptions, it is possible for such dilations to exist; see for example \cite{Brehmer,Curto-Vasilescu,Nagy-Foias,Barik-TAMS,Barik}. In \cite{Curto-Vasilescu}, it is shown that any tuple of doubly commuting pure contractions can be dilated to a tuple of doubly commuting pure isometries (see also \cite{TB-EKN-JS}). 
 
 In this paper, in \cref{Sec-DC-Dil}, we establish that any tuple of doubly commuting pure contractive semigroups can be dilated to a tuple of doubly commuting pure isometric semigroups, and that  the obtained dilation is minimal. \cref{Sec-DC-Dil} also presents a complete model for $n$-tuples of doubly commuting isometric semigroups. Although the model for a certain class in the case $n=2$ is described in \cite{DCDDC}, we present the generalization here because the techniques employed differ from those in \cite{DCDDC}. The tuples of commuting unitary semigroups constitute a subclass of the tuples of doubly commuting isometric semigroups (by \cite{Fug}). In \cref{Sec-model-normal-sgp}, we utilize the spectral theorem to establish a model for these tuples. More generally, a model for the tuples of commuting normal contractive semigroups is presented.\\
 
In the remainder of the introduction, we will present the necessary definitions, terminologies, and some well-known results that are frequently used but may not be explicitly mentioned in the main sections.

All Hilbert spaces considered in this note are separable and over the complex field $\C.$ The semigroups considered are one-parameter strongly continuous semigroups consisting of contractions on a Hilbert space. 
Let $\Tc=(T_t)_{t\ge 0}$ be a contractive semigroup on a Hilbert space $\H.$ If $A$ is the generator of $\Tc,$ then $1$ belongs to the resolvent of $A,$ i.e., $(A-I)^{-1}$ is a bounded operator on $\H.$ The Cayley transform $T:=(A+I)(A-I)^{-1}$ of $A$  is called the cogenerator of $\Tc.$ The cogenerator $T$ is a contraction and it does not have $1$ as an eigenvalue.  
Conversely, if 
$T$ is a contraction,  which does not have $1$ as an eigenvalue, then it is the cogenerator of a contractive semigroup. In fact, for such a contraction the strong limit $\slim_{r\to 1^{-}} \varphi_t(rT)$ exists for all $t\ge0,$ where $\varphi_t\in H^\infty(\D)$ is the singular inner function defined as $\varphi_t(z)=e^{t\frac{z+1}{z-1}}$ for $z\in \D.$
 Define $\varphi_t(T)$ as
\begin{equation*}
	\varphi_t(T):=\slim_{r\to 1^-}\varphi_t(rT). 
\end{equation*}
Then, $(\varphi_t(T))_{t\ge 0}$ is the semigroup of contractions with cogenerator $T.$

As any semigroup $\Tc$ is uniquely determined by its cogenerator $T$, we have $T_t= \varphi_t(T).$
We can determine the cogenerator $T$ of a semigroup $\Tc$ without explicitly delving into the generator, viz \begin{equation*}
	T=\slim_{t\to 0^+}\phi_t(T_t)
\end{equation*}
where  $\phi_t\in H^\infty(\D)$ is given by $\phi_t(z)=\frac{z-1+t}{z-1-t}$ for $z\in \D$ and $t\ge 0.$

Let $\Tc=(T_t)_{t\ge 0}$ be a semigroup. It is said to be  {\em self-adjoint, normal, unitary, or an isometric semigroup} if  each $T_t$ is  self-adjoint, normal, unitary, or an isometry, respectively.
	
Let $\Tc=(T_t)_{t\ge 0}$ be a contractive semigroup on $\H$. Then $\Tc$ is said to be  {\em pure} if $T_t^*h\to 0 $ as $t\to \infty$ for all $h\in \H.$

If $\Tc=(T_t)_{t\ge 0}$ is a contractive semigroup with the cogenerator $T.$ Then 
\begin{enumerate}
	\item $\Tc^*=(T_t^*)_{t\ge 0}$ is the contractive semigroup with  cogenerator $T^*.$ 
	\item \label{tnh}
		$\lim_{t\to \infty}\norm{T_t^*h}=\lim_{n\to \infty}\norm{{(T^*)}^nh}\text{ for all }h\in \H.$
	\item $\Tc$ is pure if and only if $T$ is pure (this follows from \cref{tnh}).
    \item  $\Tc$ is self-adjoint, normal, unitary, or an isometric semigroup if and only if $T$ is self-adjoint, normal, unitary, or an isometry, respectively.
\end{enumerate}
We refer the reader to \cite{Nagy-Foias} for details of the above. 

%Cooper's result (\cite{Cooper}) shows that any c.n.u. isometric semigroup is pure.

A tuple $(T_1,T_2,...,T_n)$ of commuting bounded operators on $\H$ is said be {\em doubly commuting} if 
\[T_iT_j^*=T_j^*T_i\text{ for all }i,j\in\{1,2,...,n\}\text{ and }i\ne j.\]

 Let $\Tc_1=(T_{1,t})_{t\ge 0}$ and $\Tc_2=(T_{2,t})_{t\ge 0}$ be two semigroups on a Hilbert space. They are said to be {\em commuting} if \begin{equation*}
 	T_{1,t}T_{2,t}=T_{2,t}T_{1,t} \text{ for all }t\ge 0.
 \end{equation*}
 
An $n$-tuple $(\Tc_1,\Tc_2,...,\Tc_n)$ of  semigroups is said to be {\em doubly commuting} if 
 $\Tc_1,\Tc_2,...,\Tc_n$ are  commuting and 
 \begin{equation*}
 	T_{i,t}T_{j,t}^*=T_{j,t}^*T_{i,t} \text{ for all }t\ge 0, 
 \end{equation*}
where $i,j\in \{1,2,...,n\}$ and $i\ne j.$ An $n$-tuple of semigroups is (doubly) commuting if and only if their cogenerators are (doubly) commuting. 

Let $\H$ be a Hilbert space and $(\Tc_1,\Tc_2,...,\Tc_n)$ be a tuple of semigroups on $\H.$ A closed subspace $\H_0$  of $\H$ is said to be a joint invariant/reducing subspace for $(\Tc_1,\Tc_2,...,\Tc_n)$ if $\H_0$ is invariant/reducing for each $\Tc_j,$ i.e., $\H_0$ is a invariant/reducing subspace for $T_{j,t}$ for each $t\ge 0$ and $ j=1,2,...,n.$

A tuple $(\Tc_1,\Tc_2,...,\Tc_n)$ of contractive semigroups on $\H$ is said to be {\em completely non-unitary (c.n.u.)} if there is no non-zero subspace $\H_0$ of $\H$ so that it is jointly reducing for $(\Tc_1,\Tc_2,...,\Tc_n)$ and $\Tc_j|_{\H_0}$ is a unitary semigroup for each $j=1,2,...,n.$

An $n$-tuple $(\V_1,\V_2,...,\V_n)$ of semigroups  on $\H$ is said to be {\em jointly unitarily equivalent} to an $n$-tuple $(\Wc_1,\Wc_2,...,\Wc_n)$ of semigroups on $\K$ if there is a unitary $\mathcal{U}:\H\to \K$ such that
\begin{equation*}
	W_{j,t}=\mathcal U V_{j,t}\mathcal U^*\text{ for all }t\ge 0, j=1,2,...,n,
\end{equation*}
where $\V_j=(V_{j,t})_{t\ge 0}$ and $\Wc_j=(W_{j,t})_{t\ge 0}.$

\section{Model for pure contractive semigroups}\label{Sec-model-pure-cont-sgp}
 In \cite{Deddens}, a model for contractive semigroups is provided. In this section, we establish a model that bears analogy to Sz-Nagy's model of a pure contraction. We start with a brief introduction to the right shift semigroup.
 
 Let $\E$ be a Hilbert space. The right shift semigroup $\S^\E=(S^\E_t)_{t\ge 0}$  on $L^2(\R_+,\E)$ is defined by 
 \[(S_t^\E f)x=\begin{cases}
 	f(x-t) &\text{if } x\ge t,\\
 	0 & \text{otherwise,}
 \end{cases}\]
 for $f\in L^2(\R_+,\E).$ The right shift semigroup is a c.n.u. semigroup of isometries.  J. L. B. Cooper (in \cite{Cooper}) showed that any c.n.u. semigroup of isometries is unitarily equivalent to the  right shift semigroup $\S^\E$ for some Hilbert space $\E.$ In particular, an isometric semigroup being c.n.u. is equivalent to it being pure.
 
 Let $S^\E$ denote the  cogenerator  of the right shift semigroup $\S^\E.$ Then $S^\E$ is a pure isometry.    An explicit expression of $S^\E$ can be found in \cite[p. 153]{Nagy-Foias}.
 Let $\z^\E\in \Hinf$  denote the operator valued function  $\z^\E(w)=wI_\E$  for $w\in \D$ where $I_\E$ denotes the identity operator on $\E.$ Define the unitary  $W_\E:L^2(\R_+,\E)\to \HDE$   by
 \begin{equation}\label{shift}
 		W_\E({(S^\E)}^n(\sqrt{2}e^{-x}\xi))=\xi\z^n\text{ for } \xi\in \E, n\ge 0.
 \end{equation}
 Then $W_\E S^\E W_\E^*=\MzE,$ that is, the cogenerators $S^\E$ and $\MzE$ are unitarily equivalent by the unitary $W_\E.$ Hence the corresponding semigroups $\S^\E$ and $(M_{\varphi_t\circ \z^\E})_{t\ge 0}$ are equivalent by the unitary $W_\E$, that is,
 	$W_\E S_t^\E W_\E^*=M_{\varphi_t\circ \z^\E} \text{ for all } t\ge 0,$
 where $\varphi_t\in H^\infty(\D)$ is the singular inner function  $\varphi_t(w)=e^{t\frac{w+1}{w-1}}$  and $\varphi_t\circ \z^\E(w):=\varphi_t(\z^\E(w))$ for $w\in \D.$ See \cite{Fact} for the details of the above. Since $S^\E$ is a pure isometry, it exhibits numerous invariant subspaces by Beurling-Lax-Halmos Theorem. When $\E=\C$, we typically omit the superscript/subscript $\E$.
  The following lemma is straightforward from the relationship between a semigroup and its cogenerator, and it could be encountered in the literature.
  
\begin{lemma}\label{inv/red}  Let $\Tc$ be a contractive semigroup on $\H$ and $T$ be the cogenerator of $\Tc.$ Let $\H_0 $ be a closed subspace of $\H.$ Then, $\H_0$ is an invariant/reducing subspace of $\Tc$ if and only if $\H_0$ is an invariant/reducing subspace of $T.$
Furthermore, if $\H_0$ is invariant under $\Tc$ (equivalently, under $T$), $\Tc|_{\H_0}:=(T_t|_{\H_0})_{t\ge 0}$ is the contractive semigroup with the cogenerator $T|_{\H_0}.$
 \end{lemma}
 Since $(M_{\varphi_t\circ \z^\E})_{t\ge 0}$ is the contractive semigroup with cogenerator $\MzE,$  it follows that $((M_{\varphi_t\circ \z^\E})^*)_{t\ge 0}$ is the contractive semigroup with  cogenerator ${(\MzE)}^*.$ The following lemma provides a collection of examples of pure contractive semigroups.
 \begin{lemma}\label{lem:comp-sgp}
 Let $\E$ be a Hilbert space and  $\Q$ be an invariant subspace for ${(\MzE)}^*$ in  $H^2(\D,\E).$ Then $(P_\Q M_{\varphi_t\circ \z^\E}|_\Q)_{t\ge 0}$ is a pure contractive semigroup, and its cogenerator is $P_\Q \MzE|_\Q.$
 \end{lemma}
\begin{proof}
  Let $\Q$ be an invariant subspace of ${(\MzE)}^*$ in  $H^2(\D,\E).$  Then by \cref{inv/red}, $\Q$ is invariant under $((M_{\varphi_t\circ \z^\E})^*)_{t\ge 0}.$ By \cref{inv/red}, $((M_{\varphi_t\circ \z^\E})^*|_\Q)_{t\ge 0}$ is the contractive semigroup with the cogenerator ${(\MzE)}^*|_\Q.$ Therefore, $(P_\Q M_{\varphi_t\circ \z^\E}|_\Q)_{t\ge 0}$ is the contractive semigroup with cogenerator $P_\Q \MzE|_\Q.$ The semigroup $(P_\Q M_{\varphi_t\circ \z^\E}|_\Q)_{t\ge 0}$ is pure because $P_\Q \MzE|_\Q$ is pure.
 \end{proof}
  
  %Recall that  $W_\E:L^2(\R_+,\E)\to H^2(\D,\E)$ defined in \cref{shift} is a unitary which intertwines $S^\E$ and $\MzE.$ i.e., $W_\E S^\E =\MzE W_\E.$ 
  Let $\Q\subseteq H^2(\D,\E).$  Then,  $\Q$ is an invariant subspace for ${(\MzE)}^*$ if and only if $W_\E^*(\Q)$ is an invariant subspace for ${(S^\E)}^*,$ where $W_\E$ is as defined in \cref{shift}.
  
  Let $\Q$ be an invariant subspace for ${(\MzE)}^*$ and $\M=W_\E^*(\Q).$ Let
  $\Gamma_\Q:\M\to \Q$ be the unitary $\Gamma_\Q:=W_\E|_{\M}.$ Then $\Gamma_\Q {(S^\E)}^*|_{\M}  \Gamma_\Q^*={(\MzE)}^*|_\Q.$ Therefore, by \cref{inv/red} and \cite[Lemma 2.4]{Fact}, $\Gamma_\Q {(S_t^\E)}^*|_{\M}  \Gamma_\Q^*={(M_{\varphi_t\circ \z^\E})}^*|_\Q$ for all $t\ge 0.$ This implies that 
  \[\Gamma_\Q [P_{\M} S_t^\E|_{\M}]  \Gamma_\Q^*=P_\Q M_{\varphi_t\circ \z^\E}|_\Q\text{ for all }t\ge 0.\]
  
  The above discussion shows that (i) $(P_\M S_t^\E|_\M)_{t\ge 0}$ is a pure contractive semigroup if $\M$ is an invariant subspace for ${(S^\E)}^*$ and (ii) A semigroup of the form $(P_\Q M_{\varphi_t\circ \z^\E}|_\Q)_{t\ge 0}$ where $\Q$ is an invariant subspace for $\MzE^*,$ is unitarily equivalent to a semigroup $(P_\M S_t^\E|_\M)_{t\ge 0}$ for some ${(S^\E)}^*$ invariant subspace  $\M,$ and vice versa.
%  Next we show that any pure contractive semigroup is unitarily equivalent to $(P_\Q M_{\varphi_t\circ \z^\E}|_\Q)_{t\ge 0}$ for some invariant subspace $\Q$ of ${(\MzE)}^*.$ 

\begin{thm}\label{pure-semigp-model}
Let $\Tc$ be a contractive semigroup on a Hilbert space $\H.$
The following are equivalent:
	\begin{enumerate}
		\item $\Tc$ is pure.
		\item There exist a Hilbert space $\E$ and an invariant subspace $\Q$ for ${(\MzE)}^*$ in $H^2(\D,\E)$ such that $\Tc$ is unitarily equivalent to $(P_\Q M_{\varphi_t\circ \z^\E}|_\Q)_{t\ge 0}.$ %\textcolor{red}{\[T_t\simeq P_\Q M_{\varphi_t\circ \z^\E}|_\Q\text{ for all }t\ge 0.\]}
	    \item There exist a Hilbert space $\E$ and an invariant subspace $\M$ for ${(S^\E)}^*$ in $L^2(\R_+,\E)$ such that $\Tc$ is unitarily equivalent to $(P_\M S_t^\E|_\M)_{t\ge 0}.$
	    %\[T_t\simeq P_\M S_t^\E|_\M\text{ for all }t\ge 0.\]
\end{enumerate}
\end{thm}
\begin{proof}
	The equivalence between (2) and (3), as well as the implication from (2) to (1), has been previously discussed above the theorem's statement. We now proceed to demonstrate that (1) implies (2).
	
	 Let $\Tc$ be a pure contractive semigroup on $\H.$ Let $T$ be the cogenerator of $\Tc.$ Then, $T$ is a pure contraction. Therefore, there exist a Hilbert space $\E$ and an invariant subspace $\Q$ of ${(\MzE)}^*$ in $H^2(\D,\E)$ such that $T$ is unitarily equivalent to $P_\Q\MzE |_\Q.$ %The contractive semigroup with cogenerator ${(\MzE)}^*$ is $((M_{\varphi_t\circ \z^\E})^*)_{t\ge 0}.$ Since $\Q$ is invariant under ${(\MzE)}^*$, by \cref{inv/red}, $\Q$ is invariant under $((M_{\varphi_t\circ \z^\E})^*)_{t\ge 0}.$ By \cref{inv/red}, the cogenerator of the contractive semigroup $((M_{\varphi_t\circ \z^\E})^*|_\Q)_{t\ge 0}$ is ${(\MzE)}^* |_\Q.$ As the cogenerators $T^*$ and ${(\MzE)}^* |_\Q$ are unitarily equivalent, the semigroups $\Tc^*$ and  $((M_{\varphi_t\circ \z^\E})^*|_\Q)_{t\ge 0}$ are unitarily equivalent by \cite{Fact}. 
	 Hence, by \cref{lem:comp-sgp} and \cite[Lemma 2.4]{Fact},  $\Tc$ is unitarily equivalent to $(P_\Q M_{\varphi_t\circ \z^\E}|_\Q)_{t\ge 0}.$
  \end{proof}
We remark that, in the equivalence between (2) and (3) of the above theorem, the Hilbert space $\E$ can be chosen to be the same space. Moreover, in the proof of (1) implies (2), $\E$ can be taken as  $\mathcal D_{T^*}.$
  
\section{The commutant of a pure contractive semigroup}\label{Sec-Comm-pure-cont-sgp}
In the previous section, we have seen a model for the pure contractive semigroups. In this section, we obtain the commutant of a pure contractive semigroup (see \cref{comm-pure-semigp}). 

 For a bounded operator $T,$ let $\sigma_p(T)$ denote the set of all eigenvalues of $T.$ For a Hilbert space $\E$, we define the class
  \begin{equation*}
  	\mathcal{C}_\E:=\left\{\psi\in H^\infty(\D,\B(\E))\text{ with }\norm{\psi}_\infty\le 1\text{ and }1\notin\sigma_p(\psi(z)) \text{ for any }z\in \D \right\}.
  \end{equation*}
For $\psi\in\CE,$ the function $\varphi_t\circ\psi$ is defined as $\varphi_t\circ\psi(z)=\varphi_t(\psi(z))$ for all $z\in\D,$ where $\varphi_t(w)=e^{t\frac{w+1}{w-1}}$ for $t\ge0$ and $w\in \D.$ It is easy to see that $\varphi_t\circ\psi\in \Hinf$ with $\norm{\varphi_t\circ\psi}_\infty\le 1,$ for all $t\ge 0.$ 
 
\begin{lemma}\label{comp-semigp}
Let $\psi\in \CE$ and let $\Q$ be a closed subspace in  $H^2(\D,\E).$ Then $P_\Q M_\psi|_\Q$ is the cogenerator of a contractive semigroup. Further, if $\Q$ is invariant under $M_\psi^*,$  then $(P_\Q M_{\varphi_t\circ \psi}|_\Q)_{t\ge 0}$ is the contractive semigroup whose cogenerator is $P_\Q M_\psi|_\Q.$
\end{lemma}
\begin{proof}	Since $\psi\in \CE,$ we have $1\notin\sigma_p(M_\psi)$ and hence $1 \notin\sigma_p(P_\Q M_\psi|_\Q).$ Therefore, $P_\Q M_\psi|_\Q$ is the cogenerator of a contractive semigroup.
	
 By \cite{Fact}, $(M_{\varphi_t\circ \psi}^*)_{t\ge 0}$ is the contractive semigroup, whose cogenerator is $M_\psi^*.$ Suppose $\Q$ is invariant under $M_\psi^*$, then it is invariant under $(M_{\varphi_t\circ \psi}^*)_{t\ge 0}.$ The cogenerator of the contractive semigroup $(M_{\varphi_t\circ \psi}^*|_\Q)_{t\ge 0}$ is $M_\psi^*|_\Q.$ Hence,  $(P_\Q M_{\varphi_t\circ \psi}|_\Q)_{t\ge 0}$ is the contractive semigroup with cogenerator $P_\Q M_\psi|_\Q.$
 \end{proof} 
Let $\Q$ be an invariant subspace of $\MzE^*$. Consider the pure contractive semigroup $(P_\Q M_{\varphi_t\circ \z^\E}|_\Q)_{t\ge 0}.$ If $\psi\in\CE,$ with the property that $\Q$ is invariant under $M_\psi^*$, then the cogenerator $P_\Q M_\psi|_\Q$ commutes  with the cogenerator $P_\Q \MzE|_\Q$. Hence the semigroup $(P_\Q M_{\varphi_t\circ \psi}|_\Q)_{t\ge 0}$ commutes with $(P_\Q M_{\varphi_t\circ \z^\E}|_\Q)_{t\ge 0}.$ We show that these are the only contractive semigroups which commute with $(P_\Q M_{\varphi_t\circ \z^\E}|_\Q)_{t\ge 0}.$ To show that we use the following lemma.
\begin{lemma}\label{main:lemma}
	Let $\eta\in H^\infty(\D,\B(\E))$ with $\norm{\eta}_\infty\le 1.$ Let $\Q\subseteq H^2(\D,\E)$ be an invariant subspace of $M_\eta^*.$ Suppose $1\notin\sigma_p(P_\Q M_\eta|_\Q).$ Then, there exists a $\psi \in H^\infty(\D,\B(\E))$  such that  $\psi\in\CE$, $\Q$ is invariant under $M_\psi^*$ and $P_\Q M_\psi|_\Q=P_\Q M_\eta|_\Q.$
\end{lemma}
\begin{proof}
	By \cite[Theorem 4.(1)]{Brown-Douglas} and the fact that $\norm{\eta}_\infty\le 1,$ the space $\E$ decomposes into the direct sum $\E=\E_1\oplus\E_2$ and $\eta$ has the form 
	\begin{equation*}
		\eta(z)=\begin{pNiceMatrix}[first-row,last-col,nullify-dots]
			\E_1 & \E_2 & \\
			I &0 & \E_1\\ 0& \theta(z) & \E_2
		\end{pNiceMatrix} \quad \text{ for all }z\in \D,
	\end{equation*}
for some $\theta\in \mathcal{C}_{\E_2},$	where $\E_1$ is the eigenspace  for $\eta$ corresponding to the eigenvalue $1.$ Hence
	\begin{equation*}
	M_\eta=\begin{pNiceMatrix}[first-row,last-col,nullify-dots]
		H^2(\D,\E_1) & H^2(\D,\E_2) & \\
		I &0 & H^2(\D,\E_1)\\ 0& M_\theta & H^2(\D,\E_2).
	\end{pNiceMatrix}
\end{equation*}
Since $\theta\in\mathcal{C}_{\E_2},$ $1\notin\sigma_p(M_\theta).$ Therefore, $H^2(\D,\E_1)$ is the eigenspace for $M_\eta$ corresponding to the eigenvalue $1.$ Now since $M_\eta$ is a contraction, the space $H^2(\D,\E_1)$ is also the eigenspace for $M_\eta^*$ corresponding to the eigenvalue $1$; see \cite[Chap. I, Proposition 3.1]{Nagy-Foias}. 

Now suppose $\Q\subseteq H^2(\D,\E)$ is an invariant subspace for $M_\eta^*$ and  $1\notin\sigma_p(P_\Q M_\eta|_\Q).$ We shall show that $\Q\subseteq H^2(\D,\E_2).$ 

Let $f\in H^2(\D,\E_1).$ Then $M_\eta^*(f)=f.$
Let $f=f_1+f_2$ where $f_1\in \Q$ and $f_2\in \Q^\perp.$ Since $\Q$ is an invariant subspace of $M_\eta^*,$ we have  \begin{equation}\label{PQMpsi}
	P_{\Q^\perp}M_\eta^*(f_2)=f_2.
\end{equation} Hence
\begin{equation*}
	\norm{f_2}=\norm{P_{\Q^\perp}M_\eta^*(f_2)}\le \norm{M_\eta^*(f_2)}\le \norm{f_2}.
\end{equation*}
This implies that $M_\eta^*(f_2)\in \Q^\perp$ and hence $M_\eta^*(f_2)=f_2$ by \cref{PQMpsi}. Therefore, $f_2\in H^2(\D,\E_1)$ and hence $f_1\in H^2(\D,\E_1).$ Thus $ M_\eta^*(f_1)=f_1.$ As $1\notin\sigma_p(P_\Q M_\eta|_\Q),$ $1\notin\sigma_p(M_\eta^*|_\Q).$ Hence we must have $f_1=0.$ Thus $f=f_2\in \Q^\perp.$ This shows that $H^2(\D,\E_1)\subseteq \Q^\perp$,  equivalently $\Q\subseteq H^2(\D,\E_2).$

Now define $\psi:\D\to \B(\E)$ by
	\begin{equation*}
	\psi(z)=\begin{pNiceMatrix}[first-row,last-col,nullify-dots]
		\E_1 & \E_2 & \\
		\kappa(z) &0 & \E_1\\ 0& \theta(z) & \E_2
	\end{pNiceMatrix} \quad \text{ for }z\in \D,
\end{equation*}
where $\kappa$ is any function from the class $\mathcal{C}_{\E_1}$ (note that $\psi$ is not  unique if $\E_1\ne \{0\}$). Then it is easy to see that $\psi\in \CE.$ As $M_\psi^*|_{H^2(\D,\E_2)}=M_\theta^*=M_\eta^*|_{H^2(\D,\E_2)},\Q\subseteq H^2(\D,\E_2)$ and $\Q$ is invariant under $M_\eta^*,$ we note that $\Q$ is invariant under  $M_\psi^*$ and $M_\psi^*|_\Q=M_\eta^*|_\Q.$ This implies that $P_\Q M_\psi|_\Q=P_\Q M_\eta|_\Q.$
\end{proof}
In the above lemma we get uncountably many $\psi\in \CE$ if $\eta\notin \CE.$ We are now ready to the theorem.
\begin{thm}\label{comm-pure-semigp}
	Let $\Q$ be an invariant subspace of $\MzE^*$ in $H^2(\D,\E).$ Let $\Tc$ be a contractive semigroup on $\Q$ which commutes with $(P_\Q M_{\varphi_t\circ \z^\E}|_\Q)_{t\ge 0}.$ Then, there exists a $\psi\in \CE$ such that $\Q$ is invariant under $M_\psi^*$ and $T_t=P_\Q M_{\varphi_t\circ \psi}|_\Q$ for $t\ge 0.$	
\end{thm}
\begin{proof}
	Let $T$ be the cogenerator of $\Tc.$  Then $T$ commutes with $P_\Q \MzE|_\Q$. Therefore, by \cref{pure-cont-comm}, there exists a $\eta\in H^\infty(\D,\B(\E))$ with $\norm{\eta}_\infty\le 1$ and $M_\eta^*(\Q)\subseteq \Q$  such that $T=P_\Q M_{\eta}|_\Q.$   Since $T$ is a cogenerator,  $1\notin\sigma_p(P_\Q M_\eta|_\Q).$ By \cref{main:lemma}, there exists $\psi\in \CE$ such that $\Q$ is invariant under $M_\psi^*$ and  $P_\Q M_\psi |_\Q=P_\Q M_\eta |_\Q.$ Thus, by \cref{comp-semigp}, $T_t=P_\Q M_{\varphi_t\circ \psi}|_\Q$ for all $t\ge 0.$	
\end{proof}
 From \cref{pure-cont-comm}, it follows that if   $(T_1,T_2)$ is a pair of commuting contractions with $T_1$ being pure, then $(T_1,T_2)$ is jointly unitarily equivalent to $(P_\Q \MzE|_\Q,P_\Q M_{\psi}|_\Q)$ for some Hilbert space $\E,$ $\psi\in H^\infty(\D,\B(\E))$ with $\norm{\psi}_\infty\le 1$ and a joint invariant subspace $\Q$ of $(\MzE^*,M_\psi^*).$ As a corollary to \cref{comm-pure-semigp}, we obtain a similar result (in \cref{cor:main}) for commuting semigroups of contractions with one of them being pure. We start with a family of prototypical examples.

Let $\E$ be a Hilbert space and  $\Q$ be a $\MzE^*$-invariant closed subspace in $H^2(\D,\E).$ Suppose $\psi_j\in \CE$ for $j=2,3,...,n$ are such that  $\Q$ is a joint invariant subspace for $(M_{\psi_2}^*,...,M_{\psi_n}^*)$ and $M_{\psi_j}^*|_\Q$'s are commuting. Then \begin{equation*}
	\Tc_1:=(P_\Q M_{\varphi_t\circ{\z^\E}}|_\Q)_{t\ge 0}, \Tc_2:= (P_\Q M_{\varphi_t\circ{\psi_2}}|_\Q)_{t\ge 0},...,\Tc_n:=(P_\Q M_{\varphi_t\circ{\psi_n}}|_\Q)_{t\ge 0}
\end{equation*} are commuting contractive semigroups on $\Q$ with $\Tc_1$ is pure. For an explicit example of such a tuple $(\Tc_1,...,\Tc_n),$ take $\E,\Q$ and $\psi_j$'s as follows: Let $\E=\E_0\oplus\E_1,$ where $\E_0$ and $\E_1$ are any two Hilbert spaces, and  $$\Q=\{a_0+a_1z:a_0,a_1\in \E\}\subseteq H^2(\D,\E).$$ Let $\{B_{2,0},B_{3,0},...,B_{n,0}\}$ be a family of commuting contractions on $\E_0$ such that $1$ is not an eigenvalue of $B_{j,0}$ for each $j.$ Let $\{B_{2,1},B_{3,1},...,B_{n,1}\}$ be a family of commuting contractions on $\E_1.$  Let $A_{j,0}$ and $A_{j,1}$ be the contractions on $\E$ given by $$A_{j,0}=B_{j,0}\oplus 0\text{ and } A_{j,1}=0\oplus B_{j,1}$$ for $j=2,...,n.$  Consider $\psi_j\in \CE$ defined by $\psi_j(z)=A_{j,0}+zA_{j,1}$ for $z\in \D$ and $j=2,3,...,n.$ (Note that $\Tc_j$ is pure if and only if $B_{j,0}$ is pure for any $j\in \{2,3,...,n\}.$)

Now we state the corollary.

%As a corollary to the above theorem we get that any $n$-tuple of commuting contractive semigroups on a Hilbert space with one of them being pure is jointly unitarily equivalent to such a concrete one. 
\begin{corollary}\label{cor:main}
	Let $(\Tc_1,\Tc_2,...,\Tc_n)$ be a tuple of commuting contractive semigroups on $\H,$ with one of them being pure, say $\Tc_1$ is pure. Then, there exist
	\begin{enumerate}
		\item  a Hilbert space $\E,$
		\item $\psi_j\in \CE$ for $j=2,...,n,$
		\item  a joint invariant subspace $\Q$ for $(\MzE^*,M_{\psi_2}^*,...,M_{\psi_n}^*)$ with the property that $M_{\psi_j}^*|_\Q$'s commute, and 
		\item a unitary $\Lambda: \H\to \Q$
	\end{enumerate}
	such that $\Lambda \Tc_1\Lambda ^*=(P_\Q M_{\varphi_t\circ \z^\E}|_\Q)_{t\ge 0}$ and $\Lambda \Tc_j\Lambda^*=(P_\Q M_{\varphi_t\circ \psi_j}|_\Q)_{t\ge 0},$  for $j=2,3,...n.$  
\end{corollary}
\begin{proof} By \cref{pure-semigp-model}, there exist a Hilbert space $\E,$ an invariant subspace $\Q$ for $\MzE^*$  in $H^2(\D,\E)$ and a unitary $\Lambda:\H\to \Q$ such that 
	\begin{equation*}
		\Lambda T_{1,t} \Lambda^*= P_\Q M_{\varphi_t\circ \z^\E}|_\Q \text{ for all } t\ge 0.
	\end{equation*}
As $\Lambda \Tc_j \Lambda^*$ commute with $\Lambda \Tc_1 \Lambda^*$ for $j=2,3,...,n,$ by \cref{comm-pure-semigp}, there exist $\psi_2,...,\psi_n\in \CE$ such that $\Q$ is a joint invariant  subspace for $(M_{\psi_2}^*,...,M_{\psi_n}^*)$ and $$\Lambda T_{j,t}\Lambda^*=P_\Q M_{\varphi_t\circ \psi_j}|_\Q\text{ for } j=2,3,...,n\text{ and }t\ge 0.$$

Now as the semigroups $\Lambda \Tc_j \Lambda^*$'s commute, their cogenerators $P_\Q M_{\psi_j}|_\Q$'s commute. This implies that  $M_{\psi_j}^*|_\Q$'s commute for $j=2,3,...,n.$	This completes the proof.
\end{proof}

\section{Model for an $n$-tuple of commuting normal contractive semigroups}\label{Sec-model-normal-sgp}
In studying tuples of commuting semigroups, a primary  objective is to identify a class of tuples of commuting semigroups with a concrete model. In \cite{Fact}, the authors, with collaborators, provided a model for c.n.u. tuples of commuting semigroups of isometries. Here, we present a model for tuples of commuting normal contractive  semigroups, utilizing the spectral theorem. We start with a lemma; the proof of the lemma is straightforward.
\begin{lemma}\label{1EV}
	Let $(X,\mu)$ be a semifinite measure space and let $\psi\in L^\infty(X,\mu).$ Then, $M_\psi$ has $1$ as an eigenvalue if and only if $E=\{x:\psi(x)=1\}$ has positive (non-zero) measure.	
\end{lemma}

%\begin{proof}
%	Suppose $M_\psi f=f$ for some non-zero $f\in L^2(X,\mu).$ This implies  $ \psi(x)f(x)=f(x)$ a.e., Since $f$ is non-zero on a positive measure set, $\psi\equiv 1$ on a positive measure set. 
	
%	Conversely, suppose the set $E=\{x:\psi(x)=1\}$ has subset $F$ with positive measure. We can choose $F$ such that that $0<\mu(F)<\infty$ (by the semifiniteness of $\mu$). Consider the characteristic function $\chi_\F\in L^2(X,\mu).$ Then $M_\psi(\chi_F)=\chi_F.$ This shows that $M_\psi$ has 1 as an eigenvalue. 
%\end{proof}

Let $(X,\mu)$ be a semifinite measure space. Let $\psi\in L^\infty(X,\mu)$ with $\norm{\psi}_\infty\le 1$ and  $\{x:\psi(x)=1\}$ has measure zero. Then, 
\begin{enumerate}
	\item For $t\ge 0,$ $\varphi_t\circ\psi$ is defined as $\varphi_t\circ\psi(x)=\varphi_t(\psi(x))$ for $x\in X,$ where $\varphi_t(w)=e^{t\frac{w+1}{w-1}}$ for $w\in \Dc\setminus\{1\}.$ Note that $\varphi_t\circ\psi\in L^\infty(X,\mu)$ with $\norm{\varphi_t\circ\psi}_\infty\le 1.$  
	\item By \cref{1EV}, $M_{\psi}$ does not have $1$ as an eigenvalue. Hence $M_{\psi}$ is  the cogenerator of a contractive semigroup, which as the semigroup of multiplication operators is as follows. 
\end{enumerate}

\begin{lemma}\label{cogen}
	Let   $(X,\mu)$ be a semifinite measure space. Let $\psi\in L^\infty(X,\mu)$ with $\norm{\psi}_\infty\le 1$ and  $\{x:\psi(x)=1\}$ has measure zero. Then,  $(M_{\varphi_t\circ \psi})_{t\geq 0}$ is the contractive semigroup on $L^2(X,\mu),$ whose cogenerator is $M_\psi.$
\end{lemma}

\begin{proof}
	The semigroup corresponding to the cogenerator $M_\psi$ is given by:
	\begin{align*}
		\varphi_t(M_\psi)&=\slim_{r\to 1^-}e^{t(rM_\psi+I)(rM_\psi-I)^{-1}}\\
		&=\slim_{r\to 1^-}e^{t(M_{r\psi+1})(M_{r\psi-1})^{-1}}\\
		&=\slim_{r\to 1^-}e^{tM_{(r\psi+1)}M_{(r\psi-1)^{-1}}}\\
		&=\slim_{r\to 1^-}e^{M_{t(r\psi+1)(r\psi-1)^{-1}}}\\
		&=\slim_{r\to 1^-}M_{e^{t(r\psi+1)(r\psi-1)^{-1}}}\\
		&=M_{e^{t(\psi+1)(\psi-1)^{-1}}}.	
	\end{align*}
	For the last equality above, note that the strong limit $\slim_{r\to 1^-}M_{e^{t(r\psi+1)(r\psi-1)^{-1}}}$ exists and for any $f\in L^2(X,\mu),$ we have  $$\lim_{r\to 1^-}e^{t(r\psi+1)(r\psi-1)^{-1}}(x)f(x)= e^{t(\psi+1)(\psi-1)^{-1}}(x)f(x)$$ 
	 for almost every $x\in X.$
\end{proof}

Let $(X,\mu)$ be a semifinite measure space. For $j=1,2,...,n,$ let $\psi_j\in L^\infty(X,\mu)$ with $\norm{\psi_j}_\infty\le 1$ and  $\{x:\psi_j(x)=1\}$ has measure zero.  Hence by \cref{cogen}, $(M_{\varphi_t\circ \psi_j})_{t\geq 0}$ is the contractive semigroup whose cogenerator is $M_{\psi_j}.$ Since  $M_{\psi_j}$'s are commuting normal operators,
$((M_{\varphi_t\circ \psi_1})_{t\geq 0},...,(M_{\varphi_t\circ \psi_n})_{t\geq 0})$ is an $n$-tuple of commuting normal contractive semigroups; such a tuple is called a {\em model $n$-normal contractive semigroup}.

In the preceding paragraph, if the functions $\psi_j$ also satisfy the condition $|\psi_j|=1$ ($\psi_j$ is real valued) almost everywhere, then the operators $M_{\psi_j}$ become commuting unitaries (commuting self-adjoint operators). Consequently, $((M_{\varphi_t\circ \psi_1})_{t\geq 0},...,(M_{\varphi_t\circ \psi_n})_{t\geq 0})$ form an $n$-tuple of commuting unitary semigroups (commuting self-adjoint contractive semigroups). Such a tuple is referred to as a {\em model $n$-unitary semigroup $($model $n$-self-adjoint contractive semigroup$)$}.

\begin{thm}
	Any $n$-tuple of commuting normal contractive semigroups is unitarily equivalent to a model $n$-normal contractive semigroup.
\end{thm}

\begin{proof}
	Let $(\Tc_1,\Tc_2,...,\Tc_n)$ be a tuple of commuting normal contractive semigroups on $\H,$ where  $\Tc_j=(T_{j,t})_{t\ge 0},$ with the cogenerator  $T_j$ for $j=1,2,...,n.$ Then $(T_1,T_2,...,T_n)$ is a tuple of commuting normal contractions on $\H.$ By spectral theorem, there exist (i) a semifinite measure space $(X,\mu),$ (ii) $\psi_j\in L^\infty(X,\mu)$ with $\norm{\psi_j}_\infty\le 1$ for $j=1,2,...,n$ and  (iii) a unitary $\gamma:\H\to L^2(X,\mu)$ such that $$T_j=\gamma^*M_{\psi_j}\gamma \text{ for }j=1,2,...,n.$$
	Note that each $M_{\psi_j}$ is a cogenerator as $T_j$ is a cogenerator. Hence, by \cref{1EV}, $\{x:\psi_j(x)=1\}$ has measure zero for each $j.$ %Note that $(M_{\psi_1},...,M_{\psi_n})$ is an $n$-tuple of commuting normal contractions.
	Thus, by \cref{cogen} and \cite[Lemma 2.4]{Fact}, we have  
	$$T_{j,t}=\gamma^*M_{\varphi_t\circ\psi_j}\gamma \text{ for } t\ge 0 \text{ and }j=1,2,...,n.$$
	This completes the proof.
\end{proof}
From the proof of the above theorem the following corollary is immediate.
\begin{corollary}
\noindent	\begin{enumerate}\label{unitary model}
		\item 	Any $n$-tuple of commuting unitary semigroups  is unitarily equivalent to a model $n$-unitary semigroup.
		\item 	Any $n$-tuple of commuting self-adjoint contractive semigroups is unitarily equivalent to a model $n$-self-adjoint contractive semigroup.		
	\end{enumerate}
\end{corollary}

\section{Tuples of doubly commuting pure contractive semigroups}\label{Sec-DC-Dil}
In this section, we obtain a model for the $n$-tuples of doubly commuting isometric semigroups (see \cite{DCDDC}, for a model of a certain  class in the case of $n=2$) and we demonstrate that any tuple of doubly commuting pure contractive semigroups can be dilated to a tuple of doubly commuting c.n.u. isometric semigroups.  

\subsection*{Models for $n$-tuples of doubly commuting isometric semigroups}
We begin with an example of a tuple of  doubly commuting c.n.u. isometric semigroups.
\begin{example}\label{DCexample}
	For $j=1,2,...,n$ and $t\ge 0,$ let $S_{j,t}:L^2(\R_+^n)\to L^2(\R_+^n)$ be defined by
	\begin{align*}
		(S_{j,t}f)(x_1,...,x_n)&=\begin{cases}
			f(x_1,...,x_{j-1},x_j-t,x_{j+1},...,x_n) & \text{if }x_j\ge t,\\
			0 & \text{else.} 
		\end{cases}		
	\end{align*}
	Define  $\S_j:=(S_{j,t})_{t\ge 0},$ then $ \S_j$ is an c.n.u. isometric semigroup for $j=1,2,...,n.$ In fact,  $(\S_1,\S_2,...,\S_n)$ is a tuple of doubly commuting c.n.u. isometric semigroups on $L^2(\R_+^n).$
\end{example}
Under the natural unitary  $\Gamma: L^2(\R_+^n)\rightarrow\underbrace{L^2(\R_+)\otimes\cdots \otimes L^2(\R_+)}_{n-\text{times}},$  we have
\begin{equation}\label{identification}
	\Gamma \S_j\Gamma^*=I_{L^2(\R_+)}\otimes\cdots \otimes I_{L^2(\R_+)}\otimes\underset{j^{\text{th}}\text{-place}}{\S}\otimes I_{L^2(\R_+)}\otimes\cdots \otimes I_{L^2(\R_+)}
\end{equation}
 for any $j\in \{1,2,...,n\},$ where $\S$ is the right shift semigroup on $L^2(\R_+).$  
\begin{definition}
	An $n$-tuple $(\V_1,\V_2,...,\V_n)$ of semigroups is called a {\em multi-shift-semigroup} if it is jointly unitarily equivalent to $(\S_1\otimes I_\E,\S_2\otimes I_\E,...,\S_n\otimes I_\E)$ on $L^2(\R_+^n)\otimes \E$ for some Hilbert space $\E.$
\end{definition}

Next, we give a concise introduction to the joint structure of a tuple of doubly commuting isometric semigroups. M. Słociński established the joint geometric structure for a pair of doubly commuting isometries in \cite{Słociński-1980}, later generalized to $n$-tuple cases by \cite{Gaspar-Suciu}, \cite{Timotin}, and \cite{Sarkar}. The analogous geometric structure for an $n$-tuple of doubly commuting isometric semigroups is also known; see for example \cite{Binzar-Lăzureanu},   where each semigroup is indexed by an unital sub-semigroup $H_i$ of an abelian group $G_i$ such that $H_i\cap H_i^{-1}=\{1_{H_i}\}$ and $G_i=H_iH_i^{-1}.$  When focused on the group $\R$ and the semigroup  $\mathbb{R}_+$, the result is as follows and it can be proved passing to the cogenerators and using  \cite[Theorem 3.1]{Sarkar}.
For any $n\ge 2,$ $I_n$ denotes the set $\{1,2,...,n\}.$ 
\begin{thm}\label{DCS-structure-thm}
	Let $(\V_1,\V_2,...,\V_n)$ be an $n$-tuple of doubly commuting isometric semigroups on $\H.$ Then there exist $2^n$ number of joint reducing subspaces $\{\H_A:A\subseteq I_n\}$ $($counting the trivial subspace $\{0\})$ for $(\V_1,\V_2,...\V_n)$ such that \[\H=\bigoplus_{A\subseteq I_n}\H_A\] and 
	for each $A\subseteq I_n$ and $\H_A\ne \{0\},\V_i|_{\H_A}$ is a c.n.u. isometric semigroup if $i\in A$ and  unitary semigroup if $i\in I_n\setminus A.$
\end{thm}
Next, we obtain a model corresponding to each part $(\V_1|_{\H_A},...,\V_n|_{\H_A}),$ where $ A\subseteq I_n,$ of the preceding theorem. The forthcoming result provides a model for tuples of doubly commuting isometric semigroups,  where all of them are completely non-unitary (c.n.u.). The subsequent result addresses the remaining parts.
\begin{thm}\label{thm:DCcnumodel}
Let $n\ge 2.$	Let $\V_1,\V_2,...,\V_n$ be  semigroups of bounded operators on $\H$. Then, $(\V_1,\V_2,...,\V_n)$ is a tuple of doubly commuting isometric semigroups with all of them being c.n.u. if and only if $(\V_1,\V_2,...,\V_n)$ is a multi-shift-semigroup. 
\end{thm}

\begin{proof}
Let $V_i$ be the cogenerator of $\V_i$ for $i=1,2,...,n.$  Then, it follows that $(V_1,V_2,...,V_n)$ is a tuple of doubly commuting pure isometries on $\H.$ Therefore, by \cite[Theorem 3.3]{Sarkar}, there exist a Hilbert space $\E$ and a unitary  \[\Gamma:(\underbrace{H^2(\D)\otimes\cdots \otimes H^2(\D)}_{n-\text{times}})\otimes \E\to \H \] such that  for each $j=1,2,...,n,$
\[\Gamma^*V_j\Gamma=\left(I_{H^2(\D)}\otimes\cdots \otimes I_{H^2(\D)}\otimes \underset{j^{\text{th}}\text{-place}}{M_\z}\otimes I_{H^2(\D)}\otimes\cdots \otimes I_{H^2(\D)}\right)\otimes I_\E.\]  Let $W_\C:L^2(\R_+)\to H^2(\D)$ be the unitary defined in \cref{shift}.
Then,  $W_\C S W_\C^*=M_\z$, where $S$ is the cogenerator of the right shift semigroup $\S.$

Let $\Lambda: \left(L^2(\R_+)\otimes\cdots \otimes L^2(\R_+)\right)\otimes \E\to \H$ be the map $\Lambda :=\Gamma \left((W_\C\otimes\cdots\otimes W_\C)\otimes I_\E\right).$ Then for each $j=1,2,...,n,$
\[\Lambda^*V_j\Lambda=\left(I_{L^2(\R_+)}\otimes\cdots \otimes I_{L^2(\R_+)}\otimes \underset{j^{\text{th}}\text{-place}}{S}\otimes I_{L^2(\R_+)}\otimes\cdots\otimes I_{L^2(\R_+)}\right)\otimes I_\E.\]
Thus, by \cite[Lemma 2.4]{Fact},  for each $j=1,2,...,n,$ we have
\[\Lambda^*\V_j\Lambda=\left(I_{L^2(\R_+)}\otimes\cdots \otimes I_{L^2(\R_+)}\otimes \underset{j^{\text{th}}\text{-place}}{\S}\otimes I_{L^2(\R_+)}\otimes\cdots\otimes I_{L^2(\R_+)}\right)\otimes I_\E.\] Hence, the tuple  $(\Lambda^*\V_1\Lambda,...,\Lambda^*\V_n\Lambda)$ is jointly unitarily equivalent to $(\S_1\otimes I_\E,...,\S_n\otimes I_\E)$; see \cref{DCexample}. Therefore, $(\V_1,...,\V_n)$ is a multi-shift-semigroup.\medskip

The converse part is trivial.
\end{proof}
The preceding theorem provides a model for the case of $A=I_n$ in \cref{DCS-structure-thm}, and \cref{unitary model} offers a model for the case of $A=\emptyset,$  where all $\V_j$'s are unitary semigroups. Without loss of generality (by rearranging 
$\V_j$'s if necessary), the subsequent theorem furnishes a model for the cases of all other subsets $A$ of $I_n.$
\begin{thm}
	 Let $1\le m<n.$ Suppose $(\V_1,\V_2,...,\V_n)$ is a tuple of doubly commuting semigroups such that $\V_j$ is a c.n.u. isometric semigroup for $1\le j\le m$ and $\V_j$ is a unitary semigroup if $m+1\le j\le n.$ Then, there exist
	\begin{enumerate}
		\item a semifinite measure space $(X,\mu),$ and 
		\item  $\psi_j\in  L^\infty(X,\mu)$ with $|\psi_j|=1$ a.e., and  $\{x:\psi_j(x)=1\}$ has measure zero for $j=1,2,...,n-m,$
		%\item a unitary \[\Delta:\H\to L^2(\R_+^n)\otimes  L^2(X,\mu)\]
	\end{enumerate}
	such that 
the tuple $(\V_1,...,\V_n)$ on $\H$ is jointly unitarily equivalent to the tuple
	\begin{equation}\label{S-in-place}
	\left(\S_1\otimes I_{L^2(X,\mu)},...,\S_m\otimes I_{L^2(X,\mu)},I_{L^2(\R_+^m)}\otimes(M_{\varphi_t\circ\psi_1})_{t\ge0},...,I_{L^2(\R_+^m)}\otimes(M_{\varphi_t\circ\psi_{n-m}})_{t\ge 0}\right)
\end{equation} on $L^2(\R_+^m)\otimes  L^2(X,\mu).$
$($when $m=1$, we have  $\S$  in place of $\S_1$ in \cref{S-in-place}$).$
\end{thm}
\begin{proof}\noindent \texttt{Case  $(m=1).$}  A proof for this case is established using  Cooper's theorem, \cite[Theorem 3.6]{DCDDC} and \cref{unitary model}.\medskip

\noindent	\texttt{Case $(1<m).$}
By  \cref{thm:DCcnumodel}, the identification in \cref{identification} and \cite[Theorem 3.6]{DCDDC} there exist a Hilbert space $\K$ and commuting unitary semigroups $\mathfrak B_j$ on $\K$ for $j=1,2,...,n-m$ such that the tuple $(\V_1,...,\V_n)$ is jointly unitarily equivalent to 
\begin{equation*}	\left(\S_1\otimes I_\K,...,\S_m\otimes I_\K,I_{L^2(\R_+^m)}\otimes \mathfrak{B}_1,...,I_{L^2(\R_+^m)}\otimes \mathfrak{B}_{n-m}\right).
\end{equation*}
Now the proof follows by applying \cref{unitary model} to the tuple $(\mathfrak B_1,...,\mathfrak B_{n-m}).$
\end{proof} 
If we substitute normal contractive semigroups $\V_j$ for $m+1 \le j \leq n$ in place of unitary semigroups in the above theorem, it remains valid with a minor modification in the conclusion. Specifically, the $\psi_j$'s satisfy $\norm{\psi_j}_\infty\le 1$ instead of $|\psi_j| = 1$ a.e..

%Next we shall find the conditions on a subspace $\Q$ so that the space  $\Q$ is a joint invariant subspace for $(\S_1^*,...,\S_n^*)$ and  $(P_\Q\S_1|_\Q,...,P_\Q\S_n|_\Q)$ is doubly commuting. 

\subsection*{Dilation of a tuple of doubly commuting pure contractive semigroups}
Suppose $(\Tc_1,\Tc_2,...,\Tc_n)$ and $(\V_1,\V_2,...,\V_n)$ are tuples of contractive semigroups on $\H$ and $\K,$ respectively. Then $(\V_1,\V_2,...,\V_n)$ is said to be a {\em dilation} of $(\Tc_1,\Tc_2,...,\Tc_n)$ if there exists an isometry  $\Omega:\H \to \K$ such that 
\[\Omega \Tc_j^*=\V_j^*\Omega \]
 for all $j=1,2,...,n,$ that is,
\[\Omega T_{j,t}^*=V_{j,t}^*\Omega \] for $j=1,2,...,n$ and $t\ge0.$

The dilation is said to be {\em minimal } if 
\[\K=\spncl\{{(V_{1,t_1}V_{2,t_2}\cdots V_{n,t_n})\Omega(\H):t_j\in \R_+ \text{ for }1\le j\le n}\}.\]

It is easy to check that the tuple $(\V_1,\V_2,...,\V_n)$ on $\K$ is a {\em dilation} of $(\Tc_1,\Tc_2...,\Tc_n)$ on $\H$ if and only if there exists a joint invariant subspace $\Q$ of $(\V_1^*,\V_1^*,...,\V_n^*)$ in $\K$ such that the tuple  $(\Tc_1,\Tc_2,...,\Tc_n)$ of contractive semigroups is jointly unitarily equivalent to the tuple of compression semigroups $(P_\Q\V_1|_\Q,...,P_\Q\V_n|_\Q).$

For $j=1,2,...,n,$ let  $M_{\z_j}$ denote the multiplication operator by the coordinate function $\z_j$ on the Hardy space of the polydisc $H^2(\D^n).$ The main result in this subsection is \cref{dilation}. It is a generalization of a result (stated below) appeared in \cite{Curto-Vasilescu,TB-EKN-JS,Muller-Vasilescu,Sarkar-Springer}. For  a tuple $\textbf{T}=(T_1,T_2,...,T_n)$ of doubly commuting pure contractions on $\H,$ let $\mathcal{D}_{\textbf{T}^*}=\overline{\ran}({\prod_{j=1}^n(I-T_jT_j^*))^\frac{1}{2}}.$ 

\begin{thm}\label{TB-EKN-JB}
Let  $\textbf{T}=(T_1,T_2,...,T_n)$ be a tuple of doubly commuting pure contractions on $\H.$ Then $(M_{\z_1}\otimes I_{\mathcal{D}_{\textbf{T}^*}},M_{\z_2}\otimes I_{\mathcal{D}_{\textbf{T}^*}},...,M_{\z_n}\otimes I_{\mathcal{D}_{\textbf{T}^*}})$ on $H^2(\D^n)\otimes\mathcal{D}_{\textbf{T}^*}$ is a  minimal isometric dilation of $\textbf{T}.$ That is, there exist a joint invariant subspace $\Q$ for $(M_{\z_1}^*\otimes I_{\mathcal{D}_{\textbf{T}^*}},M_{\z_2}^*\otimes I_{\mathcal{D}_{\textbf{T}^*}},...,M_{\z_n}^*\otimes I_{\mathcal{D}_{\textbf{T}^*}})$ in $H^2(\D^n)\otimes\mathcal{D}_{\textbf{T}^*}$ and a unitary $\omega: \H\to \Q$ such that 
\[\omega T_j\omega^* =P_\Q (M_{\z_j}\otimes I_{\mathcal{D}_{\textbf{T}^*}})|_\Q \text{ for } j=1,2,...,n\] and
 \[H^2(\D^n)\otimes\mathcal{D}_{\textbf{T}^*}=\spncl\{({M_{\z_1}^{m_1}M_{\z_2}^{m_2}\cdots M_{\z_n}^{m_n}\otimes I_{\mathcal{D}_{\textbf{T}^*}})(\Q):m_j\in \Z_+ \text{ for }1\le j\le n}\}.\]	
\end{thm}

Let $(\Tc_1,\Tc_2,...,\Tc_n)$ be a tuple of doubly commuting pure semigroups on $\H$ (i.e., each $\Tc_j$ is a pure contractive semigroup on $\H$ and the tuple $(\Tc_1,\Tc_2,...,\Tc_n)$ is doubly commuting). %Let $T_j$ be the cogenerator of $\Tc_j$ for $j=1,2,...,n.$
 As an easy application of \cref{TB-EKN-JB} to the tuple of cogenerators, we show that the tuple $(\Tc_1,\Tc_2,...,\Tc_n)$ can be dilated to a  tuple of doubly commuting c.n.u. (pure) isometric semigroups, and the obtained dilation is minimal (\cref{dilation}). To show the minimality of the dilation, we need the following lemma.
\begin{lemma}\label{lem:min}
	Let $\Tc_j=(T_{j,t})_{t\ge 0}, j=1,2,...,n$ be  commuting contractive semigroups on $\K.$ Let $T_j$ be the cogenerator of $\Tc_j$ for $j=1,2,...,n.$   Suppose $\H$ is a closed subspace of $\K.$  Then, 
	\begin{equation*}
		\spncl\{T_{1,t_1}T_{2,t_2}\cdots T_{n,t_n}h:h\in \H,t_j\in \R_+\text{ for }j=1,2,...,n\}=\K
	\end{equation*} if and only if 
	\begin{equation*}
		\spncl\{T_1^{m_1}T_2^{m_2}\cdots T_n^{m_n}h:h\in \H,m_j\in \Z_+\text{ for }j=1,2,...,n\}=\K.
	\end{equation*} 
\end{lemma}
\begin{proof}
	Assume that $\spncl\{T_{1,t_1}T_{2,t_2}\cdots T_{n,t_n}h:h\in \H,t_j\in \R_+\text{ for }j=1,2,...,n\}=\K.$ Let $k\in \K$ be such that $$\la T_1^{m_1}T_2^{m_2}\cdots T_n^{m_n}h,k\ra =0 \text{ for all }h\in \H\text{ and } m_j\in \Z_+,j=1,2,...,n.$$
	Then, $\la p_1(T_1)p_2(T_2)\cdots p_n(T_n)h,k\ra =0$ for any polynomials $p_1,p_2,...,p_n$ and $h\in \H.$ 
	
	For $h\in \H$ and $t_1,t_2,...,t_n\in \R_+,$ we have 
	\begin{align*}
		\la T_{1,t_1}T_{2,t_2}\cdots T_{n,t_n} h,k\ra
		&=\la\varphi_{t_1}(T_1)\varphi_{t_2}(T_2)\cdots\varphi_{t_n}(T_n)h,k\ra= \lim_{r\rightarrow 1^-}\left\la \left(\prod_{j=1}^n \varphi_{t_j,r}(T_j)\right)h,k\right\ra\\
		&=\lim_{r\rightarrow 1^-}\left\la\lim_{m\to \infty}\left(\prod_{j=1}^n  p_{m,t_j,r}(T_j)\right) h,k\right\ra\\&=\lim_{r\rightarrow 1^-}\lim_{m\to \infty}\left\la\left(\prod_{j=1}^n  p_{m,t_j,r}(T_j)\right) h,k\right\ra=0,	\end{align*}
	where $\varphi_{t,r}$  is an  $H^\infty(\D)$ function defined by $\varphi_{t,r}(z)=\varphi_t(rz)$ for $z\in \D$ and $(p_{m,t,r})_{m\ge 0}$ is a sequence of polynomials such that  $p_{m,t,r}\to \varphi_{t,r}$ uniformly on $\D$ for every fixed $t,r.$ 
	Hence $k=0.$  Thus, $\spncl\{T_1^{m_1}T_2^{m_2}\cdots T_n^{m_n}h:h\in \H,m_j\in \Z_+\text{ for }j=1,2,...,n\}=\K.
	$
	
	Conversely, suppose $\spncl\{T_1^{m_1}T_2^{m_2}\cdots T_n^{m_n}h:h\in \H,m_j\in \Z_+\text{ for }j=1,2,...,n\}=\K.$ Let $k\in \K$ be such that $$\la T_{1,t_1}T_{2,t_2}\cdots T_{n,t_n} h,k\ra =0 \text{ for all }h\in \H\text{ and }t_1,t_2,...,t_n\in \R_+.$$
	Then, $\la q_1(T_{1,t_1})q_2(T_{2,t_2})\cdots q_n(T_{n,t_n})h,k\ra =0$ for any polynomials $q_1,q_2,...,q_n,$ $h\in \H$ and  for any $t_1,t_2,...,t_n\in \R_+.$ For all $h\in \H $ and $m_1,m_2,...,m_n\in \Z_+,$ we have 
	\begin{align*}
		\la T_1^{m_1}T_2^{m_2}\cdots T_n^{m_n}h,k\ra
		%&=\left\la \left(\lim_{t\to 0^+}\textcolor{red}{(\phi_{t_1}(T_{1,t_1}))^{m_1}}\cdots\lim_{t\to 0^+}(\phi_{t_n}(T_{n,t_n}))^{m_n}\right) h,k\right\ra\\
		%&=\left\la \lim_{t\to 0^+}\left((\phi_{t_1}(T_{1,t_1}))^{m_1}\cdots(\phi_{t_n}(T_{n,t_n}))^{m_n}\right) h,k\right\ra\\
		%&=\lim_{t\to 0^+}\left\la \left((\phi_{t_1}(T_{1,t_1}))^{m_1}\cdots(\phi_{t_n}(T_{n,t_n}))^{m_n}\right) h,k\right\ra\\
		&=\lim_{t\to 0^+}\left\la\left( \phi_{t}^{m_1}(T_{1,t})\phi_{t}^{m_2}(T_{2,t})\cdots\phi_{t}^{m_n}(T_{n,t})\right) h,k\right\ra\\
		&=\lim_{t\to 0^+}\lim_{m\to \infty}\left\la \left(q_{m,t,m_1}(T_{1,t})q_{m,t,m_2}(T_{2,t})\cdots q_{m,t,m_n}(T_{n,t})\right)h,k\right\ra\\&=0
	\end{align*}
	where $\phi_t^n\in H^\infty(\D)$ is defined by  $\phi_t^n(z)=(\phi_t(z))^n=(z-(1-t))^n((z-(1+t))^{-1})^n$ and $(q_{m,t,n})_{m\ge 0}$ is a sequence of polynomials such that $q_{m,t,n}\to \phi_t^n$ uniformly on $\D$ as $m\to \infty$ for each fixed $t,n.$ This shows $k=0$ to complete the proof.
\end{proof}

It is well known that the space $H^2(\D^n)$ is identified with the tensor product of the Hardy space ${{H^2(\D)}^\otimes}^n,$ via., $\z_1^{m_1}\z_2^{m_2}\cdots \z_n^{m_n}\mapsto \z^{m_1}\otimes \z^{m_2}\otimes\cdots \otimes \z^{m_n}$ for $m_1,m_2,...,m_n\in \Z_+.$ Consider  the multi-shift-semigroup $(\S_1,\S_2,...,\S_n)$ on $L^2(\R_+^n).$ Let $S_j$ denote the cogenerator of $\S_j$ for $j=1,2,...,n.$ Let $W_\C:L^2(\R_+)\to H^2(\D)$ be the unitary defined in \cref{shift}. Then,  ${{W_\C}^\otimes}^n: {{L^2(\R_+)}^\otimes}^n\to {{H^2(\D)}^\otimes}^n$ is a unitary. Using this unitary map and the natural unitaries from $ L^2(\R^n_+)\rightarrow{{L^2(\R_+)}^\otimes}^n$  and  ${{H^2(\D)}^\otimes}^n\rightarrow H^2(\D^n),$ one can see that the tuple $(S_1,S_2,...,S_n)$ on $L^2(\R^n_+)$ is jointly unitarily equivalent to $(M_{\z_1},M_{\z_2},...,M_{\z_n})$ on $H^2(\D^n).$ Now we are ready for the dilation theorem.

\begin{thm}\label{dilation}
Let $(\Tc_1,\Tc_2,...,\Tc_n)$ be a tuple of doubly commuting pure contractive semigroups on $\H.$ Then, there exists a Hilbert space $\E,$ so that $(\S_1\otimes I_\E,\S_2\otimes I_\E,...,\S_n\otimes I_\E)$ is a minimal isometric dilation of $(\Tc_1,\Tc_2,...,\Tc_n).$ That is, there exist a joint invariant subspace $\Q$ for $(\S_1^*\otimes I_\E,...,\S_n^*\otimes I_\E)$ in $L^2(\R_+^n)\otimes \E$ and a unitary $\Omega:\H\to \Q$ such that 
\[\Omega\Tc_j\Omega^*=P_\Q(\S_j\otimes I_\E)|_\Q\text{ for }j=1,2,...,n\] and
\[L^2(\R_+^n)\otimes\E=\spncl\{{(S_{1,t_1}S_{2,t_2}\cdots S_{n,t_n}\otimes I_\E)(\Q):t_j\in \R_+ \text{ for }1\le j\le n}\}.\]		
\end{thm}
\begin{proof}
	Let $T_j$ be the cogenerator of $\Tc_j$ for $j=1,2,...,n.$ Then $(T_1,T_2,...,T_n)$ is a tuple of doubly commuting pure contractions. By \cref{TB-EKN-JB}, and the fact that the tuple  $(M_{\z_1},M_{\z_2},...,M_{\z_n})$ on $H^2(\D^n)$ is jointly unitarily equivalent to  $(S_1,S_2,...,S_n)$ on $L^2(\R^n_+),$
	   there exist a Hilbert space $\E,$ a joint invariant subspace $\Q$ of $(S_1^*\otimes I_\E,S_2^*\otimes I_\E,...,S_n^*\otimes I_\E)$ in $L^2(\R^n_+)\otimes \E$ and a unitary $\Omega:\H\to \Q$ such that
	\[\Omega T_j\Omega^*=P_\Q(S_j\otimes I_\E)|_\Q \text{ for }j=1,2,...,n\]
	 and 
	\begin{equation}\label{mini}
		L^2(\R_+^n)\otimes\E=\spncl\{{(S_1^{m_1}S_2^{m_2}\cdots S_n^{m_n}\otimes I_\E)(\Q):m_j\in \Z_+ \text{ for }1\le j\le n}\}
	\end{equation}
	where $S_j$ is the cogenerator of $\S_j.$ 
	
	Since $\Q$ is invariant under the cogenerator $S_j^*\otimes I_\E,$ it is invariant under the semigroup $\S_j^*\otimes I_\E$ for all $j=1,2,...,n.$ 	Note that $P_\Q(\S_j\otimes I_\E)|_\Q$ is the contractive semigroup with the cogenerator $P_\Q(S_j\otimes I_\E)|_\Q.$ Therefore, (by \cite[Lemma 2.4]{Fact}) 
	\[\Omega \Tc_j\Omega^*=P_\Q(\S_j\otimes I_\E)|_\Q \text{ for }j=1,2,...,n.\]
By \cref{lem:min}, from \cref{mini}, we have
	\[L^2(\R_+^n)\otimes\E=\spncl\{{(S_{1,t_1}S_{2,t_2}\cdots S_{n,t_n}\otimes I_\E)(\Q):t_j\in \R_+ \text{ for }1\le j\le n}\}.\]	
This completes the proof.
\end{proof}
From the proof of the above theorem one can notice that the Hilbert space $\E$ can be taken as $\E=\mathcal{D}_{\textbf{T}^*}=\overline{\ran}({\prod_{j=1}^n(I-T_jT_j^*))^\frac{1}{2}},$ where $\textbf{T}
=(T_1,T_2,...,T_n)$ and $T_j$ is the cogenerator of $\Tc_j$.

From the above theorem we see that any tuple of doubly commuting pure contractive semigroups is jointly unitarily equivalent to a compression of a multi-shift-semigroup onto a subspace which is joint invariant for $(\S_1^*\otimes I_\E, \S_2^*\otimes I_\E,...,\S_n^*\otimes I_\E).$ However, it is not true that the compression of a multi-shift-semigroup  onto each subspace which is joint invariant for $(\S_1^*\otimes I_\E, \S_2^*\otimes I_\E,...,\S_n^*\otimes I_\E)$, is doubly commuting. Therefore, a natural question is to characterize all the joint invariant subspaces $\Q$ of $(\S_1^*\otimes I_\E, \S_2^*\otimes I_\E,...,\S_n^*\otimes I_\E)$ for which the tuple of compression semigroups  $(P_\Q(\S_1\otimes I_\E)|_\Q, P_\Q(\S_2\otimes I_\E)|_\Q,...,P_\Q(\S_n\otimes I_\E)|_\Q)$ is doubly commuting.

 When $\E=\C,$ we characterize the form of $\Q$ (in \cref{rem}) using the following result of Jaydeb Sarkar (\cite{Sarkar-Jordan}): {\em Let $\Q$ be a closed subspace of $H^2(\D^n).$  Then, $\Q$ is  a joint invariant subspace for $(M_{\z_1}^*,M_{\z_2}^*,...,M_{\z_n}^*)$ and $(P_\Q M_{\z_1}|_\Q,P_\Q M_{\z_2}|_Q,..., P_\Q M_{\z_n}|_\Q)$ is doubly commuting if and only if $Q$ has the form $\Q=\Q_1\otimes\Q_2\otimes\cdots \otimes \Q_n,$ where each $\Q_j$ is an invariant subspace of $M_{\z}^*$ in  $H^2(\D).$}  The result for $n=2$ was previously proven in \cite{I-N-S}. It is unknown whether a generalization to the vector-valued case of Jaydeb Sarkar's result exists, at least to the authors' knowledge.

 Let $\Q $ be a closed subspace of $L^2(\R_+^n).$ Firstly, note that $\Q$ is a joint invariant subspace for $(\S_1^*,\S_2^*,...,\S_n^*)$  such that $(P_\Q\S_1|_\Q,P_\Q\S_2|_\Q,...,P_\Q\S_n|_\Q)$ is doubly commuting if and only if $\Q$ is a joint invariant subspace for $(S_1^*,S_2^*,...,S_n^*)$ 
such that $(P_\Q S_1|_\Q,P_\Q S_2|_\Q,...,P_\Q S_n|_\Q)$ is doubly commuting. Secondly, recall that the tuple of cogenerators $(S_1,S_2,...,S_n)$ on $L^2(\R_+^n)$  is jointly unitarily equivalent to $(M_{\z_1},M_{\z_2},...,M_{\z_n})$ on $H^2(\D^n).$ Now using the aforementioned result of Jaydeb Sarkar we get: (recall that, we identify the spaces $L^2(\R_+^n)$ and ${{H^2(\D)}^\otimes}^n$ with ${{L^2(\R_+)}^\otimes}^n$ and $H^2(\D^n)$, respectively)
\begin{remark}\label{rem}
	Let $\Q$ be a closed subspace in $L^2(\R_+^n).$ The space  $\Q$ is a joint invariant subspace for $(\S_1^*,...,\S_n^*)$ and  $(P_\Q\S_1|_\Q,...,P_\Q\S_n|_\Q)$ is doubly commuting if and only if  $\Q$ has the form $\Q=W_\C^*(\Q_1)\otimes W_\C^*(\Q_2)\otimes\cdots\otimes W_\C^*(\Q_n)$ where each $\Q_j$ is an invariant subspace of $M_\z^*$ in $H^2(\D)$ and $W_\C$ is the unitary defined in \cref{shift}.
\end{remark}

%\noindent{\bf Acknowledgements}
%The second named author is supported by the  D S Kothari postdoctoral fellowship MA/20-21/0047. 

	\Addresses
	
\end{document}